\documentclass[11pt]{article}

\usepackage{amsmath}
\usepackage{amssymb}
\usepackage{dgstpp}
\usepackage{pproof}
\usepackage{remexpp}
\usepackage{mathrsfs}
\let\rscr=\mathscr
\let\mathscr=\relax
\let\mcal=\mathcal
\usepackage{eucal}
\let\escr=\mathcal
\let\mathcal=\relax
\usepackage{accents}

\DeclareFontFamily{U}{wncy}{}
\DeclareFontShape{U}{wncy}{m}{n}{<->wncyr10}{}
\DeclareFontShape{U}{wncy}{m}{it}{<->wncyi10}{}
\DeclareSymbolFont{UWCyr}{U}{wncy}{m}{it}
\DeclareMathSymbol{\hol}{\mathalpha}{UWCyr}{"10}

\newtheorem{theorem}{Theorem}[section]

\newtheorem{corollary}[theorem]{Corollary}
\newremark{definition}[theorem]{Definition}
\newremark{example}[theorem]{Example}
\newremark{remark}[theorem]{Remark}

\renewcommand{\a}{\alpha}

\newcommand{\bD}{\pbar{D}}

\newcommand{\bg}{\pbar{\gamma}}

\newcommand{\dc}{\dot{c}}

\newcommand{\dg}{\dot{\g}}

\newcommand{\del}[1]{\nabla_{\!#1}}
\newcommand{\ds}{\oplus} 
\newcommand{\eF}{\escr{F}}

\newcommand{\g}{\gamma}
\newcommand{\G}{\Gamma}

\newcommand{\iso}{\cong}

\newcommand{\lbs}{\raise .04ex\hbox{\rm\pounds}}
\newcommand{\lsp}{[\kern-0.15em[} 

\newcommand{\norm}[1]{\lVert#1\rVert}

\newcommand{\pbar}[1]{\begingroup\let\hidewidth\relax\accentset{\hrulefill}{#1}\endgroup}
\newcommand{\pllt}{\mcal{P}}

\newcommand{\rc}{$\check{\text{r}}$} 
\newcommand{\rev}[1]{\begingroup\let\hidewidth\relax\accentset{\leftarrow}{#1}\endgroup}
\newcommand{\rsp}{]\kern-0.15em]} 
\newcommand{\surj}{\rightarrow\kern-.82em\rightarrow}

\newcommand{\V}{\rscr{V}}
\newcommand{\ve}{\varepsilon}

\newcommand{\R}{\mathbb{R}}

\renewcommand{\H}{\rscr{H}}

\renewcommand{\S}{\lower .2ex\hbox{$\escr S$}}

\makeatletter

\@addtoreset{equation}{section}

\newcommand{\ad}[1]{\mathop{\operator@font Ad}\nolimits_{#1}}
\newcommand{\Aut}{\mathop{\operator@font Aut}\nolimits}

\newcommand{\con}{\mathop{\operator@font con}\nolimits}
\newcommand{\diag}{\mathop{\operator@font diag}\nolimits}
\newcommand{\diff}{\mathop{\operator@font Diff}\nolimits}
\newcommand{\dom}{\mathop{\operator@font dom}\nolimits}
\newcommand{\econ}{\mathop{\operator@font EConn}\nolimits}
\newcommand{\End}{\mathop{\operator@font End}\nolimits}
\newcommand{\Hom}{\mathop{\operator@font Hom}\nolimits}
\newcommand{\im}{\mathop{\operator@font im}\nolimits}
\newcommand{\ld}[1]{\mathop{\operator@font\lbs}\nolimits_{#1}\!}
\newcommand{\op}{\mathop{\operator@font Op_2}\nolimits}
\newcommand{\opa}{\mathop{\operator@font Op_2^{alt}}\nolimits}
\newcommand{\qsp}{\mathop{\operator@font QSpray}\nolimits}
\newcommand{\ric}{\mathop{\operator@font Ric}\nolimits}
\newcommand{\sode}{\mathop{\operator@font DE_2}\nolimits}
\newcommand{\tr}{\mathop{\operator@font tr}\nolimits}
\makeatother

\font\heads = cmbx12

\preprint{PR2}
\title{Horizontal Path Lifting for General Connections}
\author{Phillip E. Parker \hspace{1 em} Justin M. Ryan}
\address{Mathematics Department\\
   Wichita State University\\
   Wichita KS 67260-0033\\
   USA\\
   phil@math.wichita.edu\\
   ryan@math.wichita.edu}
\date{21 October 2013} 
\abstract{
We characterize the existence of horizontal path lifts for general
connections on arbitrary fiber bundles with a new property that
also gives fresh insight into linear and $G$-connections.
}
\msc{53C29}{53C05}

\begin{document}
\maketitle

\setcounter{page}{0}\thispagestyle{empty}\strut\vfill\eject

\section{\heads Introduction and Preliminaries}

A \emph{general connection} on a smooth fiber bundle $\pi :E \surj M$ is a
subbundle $\H$ of the tangent bundle $\pi_T:  TE \surj E$ that is
complementary to the vertical bundle $\V = \ker\pi_* = \ker T\pi$,
so that
$$
TE = \H\ds\V .
$$

Connections were originally studied on tangent bundles,
and were so-named because they connected distant
tangent spaces by means of parallel transport \cite{WY}.  Parallel
transport is defined using horizontal lifts
of paths in $M$, so what we want is
a theorem like the following for any general connection.

\medskip
\noindent{\it
Let $\g : I\to M$ be a path with $\g(0)=p$ and $\g(1)=q$.  For every $v\in
E_p$ there exists a unique horizontal lift $\bg$ such that $\bg(0)=v$
and $\bg(1)\in E_q$.}

\medskip
\noindent
Ehresmann recognized that this horizontal path lifting property
is nontrivial by including it in his definition of a
connection \cite{E}.  For this reason, general connections with this property are sometimes
called \emph{Ehresmann connections}. He further showed that this property holds if $E$ has
compact fibers,
or if $\H$ is a $G$-connection on a $G$-bundle.  Kol\'{a}\rc , Michor, and
Slov\'{a}k \cite[p.\,81]{KMS} gave a sufficient condition for a general connection
to have this property. 

The usual proof for linear connections starts this way.  Consider the
pullback bundle $\g^* E$ over $I$, let $D = \tfrac{d}{dt}$ denote the
standard vector field on $I$, let $\bD$ denote its horizonal lift to $\g^*
E$, and let $c$ denote the unique integral curve of $\bD$ with $c(0)=v$.
Then $(\g^*\pi)c$ is an integral curve of $D$ with $(\g^*\pi)c(t)=t$ on
$I$.

Unfortunately, $(\g^*\pi)c$ need \emph{not} extend over all of $I$ in
general. Figure \ref{hlift} shows a simple 1-dimensional example of this as 
a connection on $TM$.
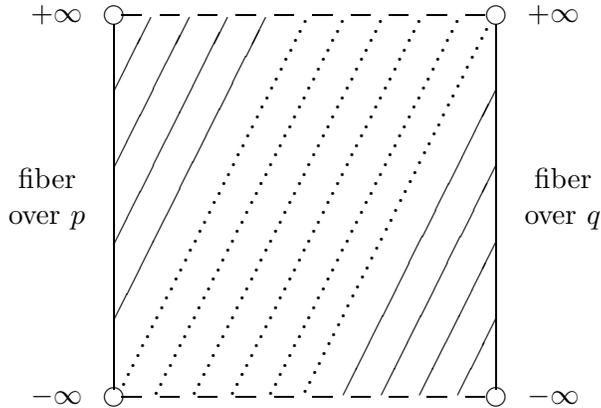
\begin{figure}[htb]
\begin{center}
\setlength{\unitlength}{.1in}
\begin{picture}(20,20)
\put(0,0){\circle{1}}
\put(0,20){\circle{1}}
\put(20,0){\circle{1}}
\put(20,20){\circle{1}}
\put(0,.4){\line(0,1){19.1}}
\put(20,.4){\line(0,1){19.1}}
\multiput(.4,0)(2,0){10}{\line(1,0){1}}
\multiput(.4,20)(2,0){10}{\line(1,0){1}}
\put(-3,0){\makebox(0,0){$-\infty$}}
\put(-3,20){\makebox(0,0){$+\infty$}}
\put(23,0){\makebox(0,0){$-\infty$}}
\put(23,20){\makebox(0,0){$+\infty$}}
\put(-6,10){\parbox{.5in}{\centering{fiber}\\ over $p$}}
\put(21,10){\parbox{.5in}{\centering{fiber}\\ over $q$}}
\put(0,16){\line(1,2){1.95}}
\put(0,12){\line(1,2){3.95}}
\put(0,8){\line(1,2){5.95}}
\put(0,4){\line(1,2){7.95}}
\multiput(.3,.6)(.25,.5){39}{.}
\multiput(2,.1)(.25,.5){40}{.}
\multiput(4,.1)(.25,.5){40}{.}
\multiput(6,.1)(.25,.5){40}{.}
\multiput(8,.1)(.25,.5){40}{.}
\multiput(9.8,.2)(.25,.5){39}{.}
\put(12,.1){\line(1,2){8}}
\put(14,.1){\line(1,2){6}}
\put(16,.1){\line(1,2){4}}
\put(18,.1){\line(1,2){2}}
\end{picture}
\bigskip
\caption{\small Horizontal lifts of a path $c$ from $p$ to $q$ in $M$.
Infinity of the fibers has been brought into the finite plane by a
compression such as $y\mapsto\tanh y$.  The solid lines begin in one fiber
and go away to infinity.  The dotted lines do not intersect either fiber.
Note that no horizontal lift of $c$ reaches from over $p$ to over $q$.}
\label{hlift}
\end{center}
\end{figure}

\begin{definition}
A general connection $\H$ is \emph{uniformly vertically bounded} (UVB) if and
only if $\H_v$ is bounded away from $\V_v$ in each $T_vE$, uniformly
along the fibers of $E$.
\end{definition}
Assuming that $\H_v$ is bounded away from $\V_v$ in each $T_v E$,
uniformly along the fibers of $E$, removes this obstacle by allowing the
standard argument to go through, preventing horizontal lifts from running
``off the edge'' \cite{Lw} or ``away to infinity'' \cite{DFN}.  All linear
connections are UVB; more generally, so are all $G$-connections for a Lie
group $G$.

First we need a way to determine how far away the $n$-plane
$\H_v$ is from the $k$-plane $\V_v$ in a single fiber $T_vE \iso \R^{n+k}$.
Wong \cite{W} showed that, given any positive-definite inner product on
$\R^{n+k}$, this can be done using $m = \min\{n,k\}$ nonzero
\emph{Wong angles}.

To compare Wong angles in different fibers of $TE$ along $E_p$, let $U
\subseteq M$ be a trivializing chart for both $E$ and $TM$ centered at $p \in
M$.  Let $g_F$ be an auxiliary Riemannian metric on the model fiber $F$ of $E$,
and $g_U$ the standard Euclidean metric on $U$.  Form the product metric
$g := g_U \times g_F$ on $E_U \iso U \times F$, and let $\del{}^g$ denote the
Levi-Civita connection for $g$.  It is well known that $\del{}^g$ determines a
system of parallel transport $\pllt^g$ on $E_U$; in particular, along $E_p$.

Fix $v \in E_p$ and let $\a :I \to E_p$ be a path with $\a(0) = w$
and $\a(1) = v$.  Denote this path by $\a :w \mapsto v$.  Let $\theta_v(w,\a)$ denote
the smallest Wong angle between the $n$-plane $\pllt^g_\a(\H_w)$ and $k$-plane
$\pllt^g_\a(\V_w)$ in the fiber $T_vE$ as measured by $g_v$.  A general connection
$\H$ is then UVB if and only if
\begin{equation}
\inf_{w \in E_p}\label{uvbinf}\left\{\inf_{\a : w\mapsto v} \theta_v(w,\a)\right\}
\geq \ve_p > 0,
\end{equation}
for all $p \in M$.

If the model fiber $F$ is parallelizable, then the situation is simplified
substantially.  In this case $T(E_U) \iso (U \times E_p) \times \R^{n+k}$, and
the Wong angles along $E_p$ can be compared unambiguously for any fixed choice
of positive-definite inner product on $\R^{n+k}$.  Note that all vector spaces and Lie
groups are parallelizable.

Now recall that linear connections satisfy $\H_{av} = a_*\H_v$ and that
$a_*$ is a motion of $T_vE$, so therefore preserves the Wong angles
along the line through 0 in $E_p$ determined by $v$ \cite{W,BN}.  Hence
the Wong angles are constant along each fiber $E_p$ so there is an absolute
minimum value among them, say $\theta_m > 0$, and this is uniform along the
fiber $E_p$.

It is easy to see that the Wong angles being constant along each fiber of a
vector bundle $E$, or \emph{vertically constant}, characterizes linear
connections. We state this explicitly since it does not seem to have been 
recorded previously.
\begin{theorem}
A\label{cwa} connection $\H$ on a vector bundle $E$ is linear if and only 
if its Wong angles are vertically constant.\eop
\end{theorem}

That $G$-connections are UVB will not be used here; a proof may be
inferred from \cite[Sec.\,31]{DFN}.

\section{\heads Main Theorem}

\begin{theorem}
Let\label{hplfb} $\pi: E \surj M$ be a smooth fiber bundle, $\H$ a general connection
on $E$, and $\g :I \to M$ a path with $\g(0) = p$ and $\g(1)=q$. For every $v
\in E_p$, there exists a unique horizontal lift $\pbar{\g} :I \to E$ such that
$\pbar{\g}(0)=v$ and $\pbar{\g}(1) \in E_q$ if and only if $\H$ is UVB.
\end{theorem}

\begin{proof}
Suppose that $\H$ has this horizontal path lifting property.  Then the horizontal
lifts of $\g$ foliate the pullback bundle $\g^*\pi :\g^*E \surj I$, and every leaf
of the foliation $\eF$ meets every fiber of $\g^*\pi$.  As $\g_* T(\eF) \subseteq \H$
and $\g$ is arbitrary, the connection $\H$ is UVB.

Conversely, assume that $\H$ is UVB.  Since $I = [0,1]$ is compact, then  $\im(\g)$
can be covered by finitely many simultaneously open charts, and we may assume
that $\im(\g)$ is contained in a single trivializing chart $U \subseteq M$ for both
$E$ and $TM$.

Consider the pullback bundle $\g^*E$ over $I$.  Since $I$ is contractible,
$\g^*E \iso I \times F$, and we may identify all fibers over $I$ with a single
copy of $F$.  Let $c \in \G(\g^*E)$ with $c(0) = v \in F$, and regard $c :I \to F$ as
a path in $F$.  Again, we may assume that $\im(c)$ lies in a single trivializing chart
$W \subseteq F$ for $TF$, and further has no self-intersections in $W$.

It now makes sense to talk about $c$ in local coordinates.  We write $\g_*\,c =
(\g,c)$ with velocity lift $(\g_*\,c)^\cdot = (\g, c, \dg, \dc)$.  Suppose further
that $\g_*\,c$ is horizontal: $(\g_*\,c)^\cdot(t) \in \H_{\g_*c(t)}$.  It remains
to show that $\g_*\,c$ is defined on all of $I$.

Horizontal spaces along $\g$ have local coordinates $(\g,y,\dg,f(y))$, where
$f :F \to \R^k$ measures the failure of $\H$ to be trivial.  Thus the lift $\g_*\,
c$ is horizontal if and only if $c$ satisfies the differential equation
\begin{equation}
\dc(t)\label{de} = f(c(t)).
\end{equation}
Since $\H$ is UVB, $\norm{f}$ is uniformly bounded above. An MVT \cite[p.\,366]{B}
implies that $\norm{c}$ is bounded on the part of $I$ where it exists.  The FEUT
\cite[pp.\,166f,\,169]{HS} gives the existence of $c$, and the Extension Theorem
\cite[p.\,171f\,]{HS} implies that $c$ extends over all of $I$.
\end{proof}

In case $E = TM$ we get the following corollary.
\begin{corollary}
Let\label{hpltb} $\H$ be a general connection on $TM$ and let $\g :I \to M$
be a path with $\g(0) = p$ and $\g(1) = q$.  For every $v \in T_pM$ there
exists a unique horizontal lift $\pbar{\g}$ such that $\pbar{\g}(0) = v$ and
$\pbar{\g}(1) \in T_qM$ if and only if $\H$ is UVB.\eop
\end{corollary}
This result was first proved by Parker directly, which inspired Ryan to 
extend it to fiber bundles.

We also obtain a notion of parallel transport in $E$.
\begin{corollary}
Each\label{pllt} path $\g$ in $M$ from $p$ to $q$ defines a diffeomorphism
$\pllt_\g :E_p \to E_q$ that we call \emph{parallel transport along $\g$}.  Note
that 
\begin{equation}
\pi_*\circ\label{pig} \dot{\pbar\g} = \dot{\g}\circ \pi :E \to TM
\end{equation}
as with vector fields.  If $\g$ is not injective, this must be interpreted
\emph{via} the pullback bundle $\g^*E$.
\end{corollary}

\begin{proof}
Uniqueness and smooth dependence of integral curves on initial conditions \cite{HS}
implies that the set of all horizontal lifts of $\g$ defines such a
diffeomorphism for each $\g$.  Equivalently, the set of all horizontal lifts
of $\g$ smoothly foliates $\g^* E$ over $I$ such that every leaf meets every
fiber of $\g^*\pi$.

Equation \eqref{pig} is obvious locally and can be extended \emph{via}
concatenation.
\end{proof}

Certain properties of parallel transport similar to those found in \cite[2.6, 9.1]{P} for
linear and $G$-connections can now be obtained for UVB connections
on fiber bundles. One major difference is of course that the maps $\pllt_\g$ might
not be linear.

\begin{corollary}
This parallel\label{pllt2} transport $\pllt_\g$ in $E$, along $\g :I \to 
M$ with $\g(0) = p$ and $\g(1) = q$, has the following properties.
\begin{enumerate}
\item Existence and uniqueness: for each $v \in E_p$ and each smooth $\g$
with $\g(0) = p$, there exists a unique smooth horizontal curve $\pllt_\g v$ from  
$v \in E_p$ to $E_q$. Alternatively, one may regard $\pllt_\g v$ as a
horizontal section of $E$ along $\g$.

\item Invertibility: allowing $v$ to vary over $E_p$, the resulting map 
$\pllt_\g :E_p \to E_q$ is a diffeomorphism. Its inverse is parallel 
transport along the reverse $\rev\g (t) := 
\g(1-t)$.

\item Parametrization independence: if $\alpha$ is a reparametrization of 
$\g$, then $\pllt_\alpha = \pllt_\g :E_p \to E_q$.

\item Smooth dependence on initial conditions: For every trivializing open set $U$
in $M$ and each smooth map $f :TU \to M$ with $f(p,0) = p$ for all $p \in U$, the
bundle map
$$
F :TU \times_U E_U \to E_U :(x,v) \mapsto \pllt_\g v,
$$
where $\g(t) := f(tx)$, is smooth.

\item Initial uniqueness: if $\alpha$ and $\g$ are two curves emanating from $p \in
M$ with $\dot\a(0) = \dg(0)$, then for every $v \in E_p$, $\pllt_\alpha v = \pllt_\g v$
have the same initial tangent vector; that is, $\dot{\pllt}_\a v(0) = \dot{\pllt}_\g v(0)$.
\eop
\end{enumerate}
\end{corollary}
In case $E$ is a vector bundle, Poor \cite{P} shows that these properties, suitably
modified, may be taken as axioms for parallel transport.

\end{document}